\documentclass{amsart}
\usepackage {hyperref}
\hypersetup {
    colorlinks = true,
    linktoc = all,
    linkcolor = blue,
    citecolor = red
}

\usepackage{verbatim}

\usepackage{amssymb}
\usepackage{url}
\usepackage{diagbox}
\usepackage{mathtools}

\usepackage{tikz}
\usetikzlibrary{automata, positioning, arrows}
\usepackage{todonotes}
\usepackage{listings}
\theoremstyle{plain}
\newtheorem{theorem}{Theorem}[section]
\newtheorem{corollary}[theorem]{Corollary}
\newtheorem{lemma}[theorem]{Lemma}

\newtheorem{conjecture}[theorem]{Conjecture}

\theoremstyle{definition}

\newtheorem{example}[theorem]{Example}

\newtheorem{remark}[theorem]{Remark}

\numberwithin{equation}{section}

\newcommand{\PP}{\mathbb{P}}
\newcommand{\Z}{\mathbb{Z}}
\newcommand{\R}{\mathbb{R}}
\newcommand{\Q}{\mathbb{Q}}
\newcommand{\F}{\mathbb{F}}

\DeclareMathOperator{\alt}{alt}
\DeclareMathOperator{\var}{var}
\DeclareMathOperator{\Gr}{Gr}
\DeclareMathOperator{\GL}{GL}

\begin{document}

\title{A finite totally nonnegative Grassmannian}


\author{John Machacek}
     \address{Department of Mathematics and Computer Science, Hampden-Sydney College, USA}
     \email{jmachacek@hsc.edu}

\subjclass[14M15]{05A15, 05E14, 14G15}



\begin{abstract}
We introduce totally nonnegative Grassmannians over finite fields where an element of a finite field is nonnegative if it is a square of an element of the finite field.
Explicit point counts are given in some special cases where we find new interpretations of sequences in the On-Line Encyclopedia of Integer Sequences (OEIS).
We compare and contrast the theory of totally nonnegative Grassmannians over a finite field with the traditional case of the field of real numbers.
\end{abstract}
\maketitle

\section{Introduction}
Postnikov~\cite{Pos} has given an elementary and fruitful framework to study total positivity of Grassmannians over the real field using the positivity of maximal minors.
The real Grassmannian is an important special case of total positivity in semisimple algebraic groups first introduced by Lusztig~\cite{Lus}.
We refer the interested reader to recent work of Bloch and Karp~\cite{BlochKarp} for a history and comparison of Postnikov's and Lusztig's approaches.
Our aim is to consider total positivity of Grassmannians over finite fields in the spirit of Cooper--Hanna--Whitlatch~\cite{posdef}.
Our results will deal with counting points in the finite field versions of totally nonnegative Grassmannians.
Galashin and Lam~\cite[Corollary 1.2]{GalashinLam} have provided point counts over finite fields of what are known as (open) positroid varieties.
This a relevant result, but not directly related to our work as positroid varieties over finite fields are defined by equations coming from positivity with real numbers while our focus is on a different idea of positivity coming from finite fields.

We will use $\F_q$ to denote the finite field with $q$ elements where $q=p^r$ for some prime $p$ and integer $r > 0$.
For an element $a \in \F_q$ we write $a \geq 0$ if $a = x^2$ has a solution in $\F_q$.
Notice $a \geq 0$ depends on the field $\F_q$; so one could want to use a notation like $a \geq_{\F_q} 0$, but we will instead simply write $a \geq 0$ as which ambient field we are using will be clear from context.
We write $a > 0$ to mean $a \geq 0$ and $a \neq 0$.
We use $a<0$ and $a\leq 0$ to mean the logical negations of $a \geq 0$ and $a > 0$ respectively.

It is not true in general that $a > 0$ implies that $-a < 0$.
For example, in $\F_5$ we have $1 > 0$ and $-1 > 0$ since $-1 = 4 = 2^2$.
Also, it is not true that this notion of positivity respects addition.
The sum of two positive elements may be negative or $0$.
Indeed, $1 > 0$ in any finite find, but $1$ generates any finite field via addition.
For example, in $\F_3$ with have $1+1 = 2 =  -1 < 0$.
However, multiplication respects this notion of positivity.
That is, $ab > 0$ if and only if $a,b > 0$ or $a,b < 0$.

We will focus throughout on the case that $p$ is odd since for any $r \geq 1$ all elements of $\F_{2^r}$ are nonnegative.
It is known that
\begin{align*}
    | \{a \in \F_q : a \geq 0\}| &= \frac{q+1}{2} & |\{a \in \F_q : a > 0\}| &= \frac{q-1}{2}
\end{align*}
for any $q = p^r$ with $p$ odd.
This is because $a \mapsto a^2$ is a homomorphism of the multiplicative group $\F_q^{\times} = \F_q \setminus \{0\}$, and in odd characteristic the kernel has size $2$.
The notations
\begin{align*}
    \F_q^{\geq 0} &= \{a \in \F_q : a \geq 0\}  & \F_q^{> 0} = \{a \in \F_q : a > 0\}\\
    \F_q^{\leq 0} &= \{a \in \F_q : a \leq 0\}  & \F_q^{< 0} = \{a \in \F_q : a < 0\}\\
\end{align*}
will be used.



A \emph{positive-definite matrix} is a symmetric $n \times n$ matrix $A$ over the real numbers $\R$ so that $v^TAv > 0$ for all column vectors $v \in \R^n$ with $v \neq 0$.
All eigenvalues being positive as well as all leading principal minors being positive are both alternative conditions that characterize positive definite matrices.
Recently, Cooper--Hanna--Whitlatch~\cite{posdef} studied positive-definite matrices over finite fields using the same notion of positivity in finite fields we have just defined.

In our present work, we will explore positivity in Grassmannians over finite fields.
First, let us review some basic properties of Grassmannians where $\F$ is any field.
For any $0 \leq k  \leq n$ we have the \emph{Grassmannian} $\Gr_{k,n}(\F) = \{V \leq \F^n : \dim(V) = k\}$ which is the set all $k$-dimensional subspaces of $\F^n$.
For any $V \in \Gr_{k,n}(\F)$ we can choose a full rank $k \times n$ matrix $A$ such that $V$ is the row span of $A$.
When this is the case we write $[A] = V$.
If $A$ and $B$ are two $k\times n$ matrices, then $[A] = [B]$ if and only if there exists $M \in \GL_k(\F)$ such that $MA = B$.
This exactly means the rows of $A$ and $B$ each form a basis of $V$ and $M$ encodes the change of basis between them.

We let $[n] = \{1,2,\dots,n\}$ and let $\binom{[n]}{k}$ denote the collection of all $k$-element subsets of $[n]$.
For any $I \in \binom{[n]}{k}$ we let $\Delta_I(A)$ denote the \emph{Pl\"ucker coordinate} that is given by the maximal minor of $A$ with columns indexed by $I$.
If $[A] = [B]$ where $MA = B$ then $\Delta_I(B) = \det(M)\Delta_I(A)$ for all $I \in \binom{[n]}{k}$.
Hence, the Pl\"ucker coordinates are well-defined when viewed as projective coordinates.
This means that we can use the notation $\Delta_I(V)$ for any $V \in \Gr_{k,n}(\F)$ and this is well defined up to simultaneous rescaling of all the Pl\"ucker coordinates by a fixed nonzero element of the field.

We now define the \emph{totally nonnegative finite Grassmannian} to be
\[\Gr_{k,n}^{\geq 0}(\F_q) = \{V \in \Gr_{k,n}(\F_q): \Delta_I(A) \geq 0 \text{ for all } I \in \binom{[n]}{k} \text{ for some $A$ with } [A] = V\}\]
for any $0 \leq k \leq n$ and finite field $\F_q$.
This is exactly the definition Postnikov~\cite{Pos} makes with $\R$ in place of $\F_q$.

\begin{remark}
We can define the \emph{totally positive finite Grassmannian} to be
\[\Gr_{k,n}^{> 0}(\F_q) = \{V \in \Gr_{k,n}(\F_q): \Delta_I(A) > 0 \text{ for all } I \in \binom{[n]}{k} \text{ for some $A$ with } [A] = V\}\]
though this is typically empty. Indeed for any $q$ and $k>1$ there exists $N$ so that $|\Gr_{k,n}^{> 0}(\F_q)| = 0$ for all $n \geq N$.
This is because with $n$ sufficiently large building a full rank $k \times n$ matrix taking columns from $\F_q^k$ we will necessarily have a repeated column causing some Pl\"ucker coordinates to vanish.
\end{remark}

If $\F_q$ has even characteristic the totally nonnegative finite Grassmannian is just the entire Grassmannian.
Also, for any $V \in \Gr_{k,n}^{\geq 0}(\F_q)$ it is the case that if $[B] = V$ then either $\Delta_I(B) \geq 0$ for all $I \in \binom{[n]}{k}$ or else $\Delta_I(B) \leq 0$ for all $I \in \binom{[n]}{k}$.
This follows from the fact that $\Delta_I(B) = \det(M) \Delta_I(A)$ along with the fact that the product of two field elements is positive if and only if both are positive or both are negative.
In particular, we may assume when convenient that $[A] = V$ and $A$ is in \emph{reduced row echelon form (RREF)}.
In this situation we must have $\Delta_I(A) = 1 > 0$ for some $I$ (which indexes the pivots of $A$), and thus $V \in \Gr_{k,n}^{\geq 0}(\F_q) $ if and only if all Pl\"ucker coordinates are nonnegative when computed via $A$.

Before proceeding, we recall a more classical notion of positivity for matrices and explain how it relates to the totally nonnegative Grassmannian.
A matrix is said to be \emph{totally nonnegative} if all of its minors are nonnegative.
This is a stronger condition than merely representing a point in the totally nonnegative Grassmannian, which only requires that the maximal minors be nonnegative.
However, there exists an embedding that maps a $k \times (n-k)$ totally nonnegative matrix to a point in the totally nonnegative Grassmannian of $k$-dimensional subspaces of an $n$-dimensional space.
This map, which can be defined over any field, takes a $k \times (n-k)$ totally nonnegative matrix and scales the $i$th row by $(-1)^{k-i}$ then prepends a $k \times k$ identity matrix to it~\cite[Lemma 3.9]{Pos}.

In Section~\ref{sec:orth} we consider orthogonal complements and use them to observe the symmetry $|\Gr^{\geq 0}_{k,n}(\F_q)| = |\Gr^{\geq 0}_{n-k,n}(\F_q)|$ for any $0 \leq k \leq n$ in Theorem~\ref{thm:alt}.
Subfields are used to construct elements of totally nonnegative finite Grassmannians in Section~\ref{sec:subfield}.
The projective space case is easy and dealt with in Section~\ref{sec:proj} where we enumerate $|\Gr^{\geq 0}_{1,n}(\F_q)|$ in Theorem~\ref{thm:proj}.
In Section~\ref{sec:planes} we look at $k=2$ and enumerate $|\Gr^{\geq 0}_{2,n}(\F_q)|$ for $q=3$ and $q=5$ in Theorem~\ref{thm:2nF3} and Theorem~\ref{thm:2nF5} respectively.
These enumeration results include formulas and generating functions as well as a connection to coefficients of Chebyshev polynomials in the case of $|\Gr^{\geq 0}_{2,n}(\F_3)|$.
In Section~\ref{sec:examples}, we give various examples demonstrating how various results that are known about the totally nonnegative real Grassmannian change when considering finite fields.

\section{Orthogonal complements}
\label{sec:orth}
In this section, we start by working with an arbitrary field.
Let $\F$ be any field, for any $0 \leq k \leq n$ there is a duality between $\Gr_{k,n}(\F)$ and $\Gr_{n-k,n}(\F)$ via orthogonal complements.
In this section, we explain how this duality can also be made to work on the totally nonnegative parts.
For any subspace $V \in \Gr_{k,n}(\F)$ so that $V = [A]$ for a full rank $k \times n$ matrix $A$ there is the subspace $V^{\perp} \in \Gr_{n-k,n}(\F)$ which can be defined to be the null space of the matrix $A$.
The map $\perp:\Gr_{k,n}(\F) \to \Gr_{n-k,n}(\F)$ is an involution (i.e. $(V^{\perp})^{\perp} = V$), and hence a bijection implying that $|\Gr_{k,n}(\F)| = |\Gr_{n-k,n}(\F)|$.
To work with the totally nonnegative part we define $\alt: \F^n \to \F^n$ by
\[\alt((v_1, v_2, v_3, v_4, \dots, v_n)) = (v_1, -v_2, v_3, -v_4, \dots, (-1)^{n-1}v_n)\]
for any $(v_1, v_2, v_3, v_4, \dots, v_n) \in \F^n$.
The map $\alt$ when $\F = \R$ was used by Karp~\cite{karp} in the study on the totally nonnegative real Grassmannian.
For a subspace $V$ we let $\alt(V) = \{\alt(v): v \in V\}$ and for a matrix $A$ we let $\alt(A)$ denote the matrix where $\alt$ has been applied to the rows for $A$.
In this way it is the case that $\alt(V) = [\alt(A)]$ when $V = [A]$.

\begin{lemma}[{\cite[Section 7]{Hoch} see also \cite[Lemma 1.11 (ii)]{karp}}]
Given any $0 \leq k \leq n$ and any field $\F$, as projective coordinates, $\Delta_{I}(\alt(V^{\perp})) = \Delta_{[n] \setminus I}(V)$ for all $I \in \binom{[n]}{n-k}$ for any $V \in \Gr_{k,n}(\F)$.
That is, if $V = [X]$ and $\alt(V^{\perp}) = [Y]$ for some chosen matrix representatives, then there exists a fixed nonzero $\alpha \in \F^{\times}$ so that the determinants $\Delta_{I}(Y) = \alpha \Delta_{[n] \setminus I}(X)$ for all $I \in \binom{[n]}{n-k}$.
\label{lem:altprep}
\end{lemma}

This lemma implies the following theorem.
In the context of our totally nonnegative Grassmannians for finite fields, it means that for fixed $n$ the sequence $(|\Gr^{\geq 0}_{k,n}(\F_q)|)_{0 \leq k \leq n}$ is palindromic.
Examples for $\F_3$ and $\F_5$ can be seen in Table~\ref{tbl:F3} and Table~\ref{tbl:F5} respectively.

\begin{table}
\begin{minipage}{.45\linewidth}
\centering
\begin{tabular}{c|ccccccc}
\diagbox{$n$}{$k$} & $0$&$1$&$2$&$3$&$4$&$5$&$6$\\ \hline
$0$ &$1$ \\
$1$ & $1$ & $1$ \\
$2$ & $1$ & $3$ & $1$ \\
$3$ & $1$ & $7$ & $7$ & $1$ \\
$4$ & $1$ & $15$ & $32$ & $15$ & $1$ \\
$5$ & $1$ & $31$ & $120$ & $120$ & $31$ & $1$ \\
$6$ & $1$ & $63$ & $400$ & $703$ & $400$ & $63$ & $1$
\end{tabular}
\caption{The cardinalities $|\Gr^{\geq 0}_{k,n}(\F_3)|$.} 
\label{tbl:F3}
\end{minipage}
\hfill
\begin{minipage}{.45\linewidth}
\centering
\begin{tabular}{c|cccccc}
\diagbox{$n$}{$k$} & $0$&$1$&$2$&$3$&$4$&$5$\\ \hline
$0$ & $1$ \\
$1$ & $1$ & $1$ \\
$2$ & $1$ & $4$ & $1$ \\
$3$ & $1$ & $13$ & $13$ & $1$ \\
$4$ & $1$ & $40$ & $122$ & $40$ & $1$ \\
$5$ & $1$ & $121$ & $1010$ & $1010$ & $121$ & $1$ \\
\end{tabular}
\caption{The cardinalities $|\Gr^{\geq 0}_{k,n}(\F_5)|$.} 
\label{tbl:F5}
\end{minipage}
\end{table}
 
\begin{theorem}
For any $0 \leq k \leq n$ and any finite field $\F_q$ the map 
\[\alt \circ \perp: \Gr^{\geq 0}_{k,n}(\F_q) \to \Gr^{\geq 0}_{n-k,n}(\F_q)\] 
is a well-defined bijection and thus 
\[|\Gr_{k,n}^{\geq 0}(\F_q)| = |\Gr^{\geq 0}_{n-k,n}(\F_q)|.\]
\label{thm:alt}
\end{theorem}
\begin{proof}
The map
\[\alt \circ \perp: \Gr_{k,n}(\F) \to \Gr_{n-k,n}(\F)\]
is a bijection with inverse
\[\perp \circ \alt: \Gr^{\geq 0}_{n-k,n}(\F) \to \Gr^{\geq 0}_{k,n}(\F)\]
since both $\perp$ and $\alt$ are involutions.
By Lemma~\ref{lem:altprep} it follows that this bijection restricts to a well-defined bijection on the totally nonnegative parts.
Hence, the cardinalities in question must be equal.
\end{proof}

We now give a natural conjecture. It is well known that the sequence $\left(|\Gr_{k,n}(\F_q)|\right)_{0 \leq k \leq n}$ is both palindromic and log-concave.
It is clear that all cardinalities are nonzero; so, log-concavity is a stronger property implying unimodality.
For more on unimodal and log-concave sequences, including explicit definitions as well as discussion of $|\Gr_{k,n}(\F_q)|$, an interested reader can see the survey~\cite{unilog}.

\begin{conjecture}
    For $q = p^r$ with $p$ an odd prime the sequence $\left(|\Gr_{k,n}^{\geq 0}(\F_q)|\right)_{0 \leq k \leq n}$ is log-concave.
\end{conjecture}

\section{Subfields}\label{sec:subfield}
It is known that $\F_{p^s}$ contains $\F_{p^r}$ as a subfield if and only if $r$ divides $s$.
Moreover, when $r$ divides $s$ the finite field $\F_{p^s}$ contains exactly one copy of $\F_{p^r}$.
So, we may view $\F_{p^r}$ as a subset of $\F_{p^s}$.
Also, each $a \in \F_q$ satisfies $a^q = a$, and thus $a^{q-1} = 1$ whenever $a \neq 0$.
Here $\F_{p^r}$ is the subset of $\F_{p^s}$ consisting of those elements $a$ satisfying $a^{p^r} = a$.
The \emph{group of units} $\F_q^{\times} = \{a \in \F_q : a \neq 0\}$ is a cyclic group of order $q-1$.
This means there exists $\alpha \in \F^{\times}_q$ called a \emph{primitive root} so that $\F_q^{\times} = \{\alpha^m : 0 \leq m < q-1\}$.

\begin{lemma}
    If $\alpha$ is a primitive root of $\F_q$ a finite field of odd characteristic, then
    \[\F_q^{> 0} = \{\alpha^{2m} : 0 \leq m < \frac{q-1}{2}\}\]
    and thus
    \[\F_q^{> 0} = \{a \in \F_q : a^\frac{q-1}{2} = 1\}.\]
    \label{lem:>0}
\end{lemma}
\begin{proof}
    It is clear that $\{\alpha^{2m} : 0 \leq m < \frac{q-1}{2}\}$ consists of $\frac{q-1}{2}$ distinct nonzero elements which are squares in $\F_q$.
    So, this set must be the set of positive elements since we know $|\F_q^{> 0}| = \frac{q-1}{2}$.
    From this the second equality in the lemma follows since $(\alpha^{2m})^{\frac{q-1}{2}} = (\alpha^{q-1})^m = 1$.
\end{proof}

\begin{lemma}
    If $s = 2mr$ for some integer $m \geq 1$ and $p$ is an odd prime, then $\F_{p^r} \subseteq \F_{p^s}^{\geq 0}$.
    \label{lem:subfield}
\end{lemma}
\begin{proof}
    Consider $a \in \F_{p^r} \subseteq \F_{p^s}$.
    If $a = 0$, then $a \in \F_{p^s}^{\geq 0}$.
    If $a \neq 0$, then $a^{p^r - 1}=1$.
    This means
    \[a^{\frac{p^s - 1}{2}} = a^{\frac{p^{2mr} - 1}{2}} = a^{(p^r-1)\left(\frac{p^{(2m-1)r} + p^{(2m-2)r} + \dots + 1}{2}\right)}=1\]
    and the lemma holds after applying Lemma~\ref{lem:>0}.
    Notice that 
    \[p^{(2m-1)r} + p^{(2m-2)r} + \dots + 1\]
    is even since each term in the sum in odd and there is an even number of terms.
\end{proof}

\begin{theorem}
    If $s = 2mr$ for some integer $m \geq 1$ and $p$ is an odd prime, then there is a natural injective map $\Gr_{k,n}(\F_{p^r}) \hookrightarrow \Gr_{k,n}^{\geq 0}(\F_{p^s})$ and thus $|\Gr_{k,n}^{\geq 0}(\F_{p^s})| \geq |\Gr_{k,n}(\F_{p^r})|$.
    \label{thm:sub}
\end{theorem}
\begin{proof}
    If $[A] \in \Gr_{k,n}(\F_{p^r})$ we can consider $[A] \in \Gr_{k,n}(\F_{p^s})$.
    Since all Pl\"ucker coordinates of $[A]$ are in $\F_{p^r}$ and $\F_{p^r} \subseteq \F_{p^s}^{\geq 0}$ by Lemma~\ref{lem:subfield} it follows $[A] \in \Gr_{k,n}^{\geq 0}(\F_{p^s})$.
\end{proof}

Generally, the bound in Theorem~\ref{thm:sub} is far from tight.
As an example, the sequence $\left(|\Gr_{k,5}^{\geq 0}(\F_9)|\right)_{0 \leq k \leq 5}$ is
\[1, 781, 13010, 13010, 781, 1\]
while the sequence $\left(|\Gr_{k,5}(\F_3)|\right)_{0 \leq k \leq 5}$ is
\[1, 121, 1210, 1210, 121, 1.\]

\section{Enumeration of positive and nonnegative lines}
\label{sec:proj}
In this section take $k=1$ and work inside $\Gr_{1,n}(\F_q) = \F_q \PP^{n-1}$ which is a projective space of dimension $n-1$.
We consider both the totally nonnegative finite Grassmannian and the totally positive finite Grassmannian as both are nonempty for all values $n \geq 1$.
\begin{theorem}
    For any $n \geq 1$ and $\F_q$ of odd characteristic
    \[|\Gr_{1,n}^{\geq 0}(\F_q)| = \sum_{i=0}^{n-1} \left(\frac{q+1}{2}\right)^i  = \sum_{i=0}^{n-1} \sum_{j = i}^{n-1} \frac{\binom{j}{i}}{2^j} q^i = \frac{\left(\frac{q+1}{2}\right)^n - 1}{\left(\frac{q+1}{2}\right) - 1}\]
    gives the point count of the totally nonnegative finite Grassmannian for $k=1$.
For any $n \geq 1$ and $\F_q$ of odd characteristic
\[|\Gr_{1,n}^{>0}(\F_q)| = \left(\frac{q-1}{2}\right)^{n-1}  = 2^{-(n-1)} \sum_{i=0}^{n-1} (-1)^{n-1-i} \binom{n-1}{i} q^i\]
    gives the point count of the totally positive finite Grassmannian $k=1$.
    \label{thm:proj}
\end{theorem}
\begin{proof}
    Take $V \in \Gr_{1,n}^{\geq 0}(\F_q)$ and $A$ in RREF with $[A] = V$.
    This means $A$ is a row vector where the first nonzero entry of $A$ is a $1$, and $A$ necessarily has some nonzero entry.
    Then every other entry in $A$ must be some nonnegative element of $\F_q$.
    When the first nonzero entry for $A$ is in position $1 \leq i \leq n$ we then have $n-i$ positions to the right for which we have $\frac{q+1}{2}$ choices for each.
    Thus,
    \[
        |\Gr_{1,n}^{\geq 0}(\F_q)| = \sum_{i=1}^n \left(\frac{q+1}{2}\right)^{n-i} =\sum_{i=0}^{n-1} \left(\frac{q+1}{2}\right)^i = \sum_{i=0}^{n-1} \sum_{j = i}^{n-1} \frac{\binom{j}{i}}{2^j} q^i
    \]
    is the point count in the totally nonnegative finite Grassmannian. The last equality follows because the first summation formula is a finite geometric series. 

        Now take $V \in \Gr_{1,n}^{> 0}(\F_q)$ and $A$ in reduced row echelon form with $[A] = V$.
        Then the first entry of $A$ must be $1$ and all other entries are chosen to be one of the $\frac{q-1}{2}$ positive elements of $\F_q$.
        This results in count of $\left(\frac{q-1}{2}\right)^{n-1}$ which can be expanded with the binomial theorem.
\end{proof}

We find in Theorem~\ref{thm:proj} that $|\Gr_{1,n}^{\geq 0}(\F_q)| \in \Z \left[\frac{q+1}{2}\right]$ and $|\Gr_{1,n}^{> 0}(\F_q)| \in \Z\left[\frac{q-1}{2} \right]$ are polynomials with nonnegative integer coefficients in the number of nonnegative and positive elements of $\F_q$ respectively.
In fact, $|\Gr_{1,n}^{\geq 0}(\F_q)|$ turns out to be the usual $q$-integer $[n]_q = 1 + q + q^2 + \cdots q^{n-1}$ evaluated at $q \to \frac{q+1}{2} = |\F^{\geq 0}|$.
Similarly, in the very special setting of $k=1$ we have that $|\Gr_{1,n}^{> 0}(\F_q)|$ is simply a power of $\frac{q-1}{2} = |\F^{> 0}_q|$.
Thus, $|\Gr_{1,n}^{\geq 0}(\F_q)|, |\Gr_{1,n}^{> 0}(\F_q)| \in \Q[q]$ are polynomials in $q$ with rational coefficients.
These polynomials take integer values at odd integer inputs.
Examples of these polynomials for the totally nonegative part are
\begin{align*}
    |\Gr_{1,2}^{\geq 0}(\F_q)| &=  \frac{1}{2} q + \frac{3}{2} \\
    |\Gr_{1,3}^{\geq 0}(\F_q)| &= \frac{1}{4} q^{2} + q + \frac{7}{4}  \\
    |\Gr_{1,4}^{\geq 0}(\F_q)| &=  \frac{1}{8} q^{3} + \frac{5}{8} q^{2} + \frac{11}{8} q + \frac{15}{8} \\
    |\Gr_{1,5}^{\geq 0}(\F_q)| &=  \frac{1}{16} q^{4} + \frac{3}{8} q^{3} + q^{2} + \frac{13}{8} q + \frac{31}{16} \\
\end{align*}
while examples of the totally positive part are
\begin{align*}
    |\Gr_{1,2}^{> 0}(\F_q)| &=  \frac{1}{2} q - \frac{1}{2} \\
    |\Gr_{1,3}^{> 0}(\F_q)| &= \frac{1}{4} q^{2} - \frac{1}{2} q + \frac{1}{4}  \\
    |\Gr_{1,4}^{> 0}(\F_q)| &=  \frac{1}{8} q^{3} - \frac{3}{8} q^{2} + \frac{3}{8} q - \frac{1}{8} \\
    |\Gr_{1,5}^{> 0}(\F_q)| &=  \frac{1}{16} q^{4} - \frac{1}{4} q^{3} + \frac{3}{8} q^{2} - \frac{1}{4} q + \frac{1}{16}.\\
\end{align*}

\section{Enumeration of nonnegative $2$-planes}
\label{sec:planes}
In this section, we look at the number of points in $\Gr_{2,n}^{\geq 0}(\F_q)$ for certain values of $q$.
We are able to obtain the point count and find generating functions in the cases of $q=3$ and $q=5$.
For $q=3$ we also find a connection to coefficients for Chebyshev polynomials.

\subsection{The field $\F_3$}
In this subsection, we deal with the field $\F_3$.
It is then the case that $1>0$ is the only positive element and $-1 < 0$ is the only negative element.
We will find the cardinality of $\Gr_{2,n}^{\geq 0}(\F_3)$ for each $n \geq 2$.
The following lemma will be used in our point counting of $\Gr_{2,n}^{\geq 0}(\F_3)$.

\begin{lemma}
    Let $a_n = \frac{2^{n-2}(n+1)(n+2)(n+6)}{3}$ for $n \geq 0$, then the following recurrence is satisfied
    \[a_n - 8a_{n-1} + 24a_{n-2} - 32a_{n-3} + 16a_{n-4} = 0\]
    for $n \geq 4$ with initial conditions $a_0 = 1$, $a_1 = 7$, $a_2 = 32$, and $a_3 = 120$.
    \label{lem:rec}
\end{lemma}
\begin{proof}
The lemma can be verified by direct computation.
\end{proof}

One interpretation we will give involves the \emph{Chebyshev polynomials of the first kind} which are defined by $T_n(\cos(\theta)) = \cos(n \theta)$ for $n \geq 0$.
These polynomials can also be defined by the recurrence $T_{n+1}(x) = 2xT_n(x) - T_{n-1}(x)$ with $T_0(x) = 1$ and $T_1(x) = x$.
We will make use of a combinatorial formula for the coefficients of $T_n(x)$.
To do this we need the notion of a \emph{tiling} which is a way to cover a $1 \times n$ strip with $1 \times 1$ \emph{squares} and $1 \times 2$ \emph{dominoes}.
It is well known that the \emph{Fibonacci number} $f_{n+1}$ counts the number of such tiling of an $1 \times n$ strip where $f_{n+1} = f_n + f_{n-1}$ for $n \geq 2$ and $f_1 = f_2 = 1$.
The \emph{weight} of any such tiling is the product of $-1$ for each domino, $x$ for a square in the leftmost position, and $2x$ for a square in any other position.
Benjamin and Walton~\cite[Theorem 3]{Cheb} have shown that the Chebyshev polynomial $T_n(x)$ is equal to the sum of the weights of all tilings of a $1 \times n$ strip.
Figure~\ref{fig:tile} shows all tilings of a $1 \times 4$ rectangle from which $T_4(x)$ can be calculated.
One can obtain a particular coefficient of $T_n(x)$ if we control the number of dominoes in our tilings.

\begin{figure}
\begin{tikzpicture}[scale=0.6]
\draw (0,0) rectangle (2,1) node[pos=.5] {$-1$};
\draw (2,0) rectangle (4,1)  node[pos=.5] {$-1$};

\draw (0,-1.5) rectangle (2,-0.5) node[pos=.5] {$-1$};
\draw (2,-1.5) rectangle (3,-0.5)  node[pos=.5] {$2x$};
\draw (3,-1.5) rectangle (4,-0.5)  node[pos=.5] {$2x$};

\draw (5,-1.5) rectangle (6,-0.5) node[pos=.5] {$x$};
\draw (6,-1.5) rectangle (8,-0.5)  node[pos=.5] {$-1$};
\draw (8,-1.5) rectangle (9,-0.5)  node[pos=.5] {$2x$};

\draw (10,-1.5) rectangle (11,-0.5) node[pos=.5] {$x$};
\draw (11,-1.5) rectangle (12,-0.5)  node[pos=.5] {$2x$};
\draw (12,-1.5) rectangle (14,-0.5)  node[pos=.5] {$-1$};

\draw (0,-3) rectangle (1,-2) node[pos=.5] {$x$};
\draw (1,-3) rectangle (2,-2)  node[pos=.5] {$2x$};
\draw (2,-3) rectangle (3,-2)  node[pos=.5] {$2x$};
\draw (3,-3) rectangle (4,-2)  node[pos=.5] {$2x$};
\end{tikzpicture}
\begin{caption}
{The $f_5 = 5$ tilings in the figure demonstrate that $T_4(x) = 1 -8x^2 + 8x^4$.}
\end{caption}
\label{fig:tile}
\end{figure}

For the statement of the next lemma, we will need notation for extracting a coefficient from any formal power series (which includes the case of polynomials).
Given a formal power series
\[f(x) = \sum_{n = 0}^{\infty} c_n x^n\]
we write $[x^n](f(x)) = c_n$ to denote the coefficient of $x^n$.

\begin{lemma}
    For any $n \geq 0$, $[x^n]\left(T_{n+6}(x)\right) = \frac{-2^{n-2}(n+1)(n+2)(n+6)}{3}$.
    \label{lem:tile}
\end{lemma}
\begin{proof}
    To find the coefficient of $x^n$ in $T_{n+6}(x)$ we look at tilings of a $1 \times (n+6)$ strip with exactly $3$ dominoes of weight $-1$.
    This means there will be $n$ squares each of weight $x$ or $2x$ thus resulting in a contribution to the $x^n$ term in the polynomial.
    First, assume the leftmost position is covered by a domino, this leaves $n+4$ squares to cover using exactly $2$ dominoes.
    There are $\binom{n+2}{2}$ tilings of this form as any such tiling must use $n$ total squares and dominoes with $2$ of them being dominoes.
    All tilings of this type contribute $-2^nx^n$ since the weight of each square in the tiling is $2x$.
    Next, assume the leftmost position is covered by a single square contributing a weight of $x$.
    This means we have $n+5$ squares remaining to cover in a way that uses exactly $3$ dominoes.
    There are $\binom{n+2}{3}$ tilings like this as we will use $n+2$ more tiles where $3$ of them must be dominoes.
    Each tiling of this type will contribute $-2^{n-1}x^n$ since we have $n$ total squares one of which has weight $x$ while the others have weight $2x$.
    As that remains is to compute,
    \[-2^n \binom{n+2}{2} x^n - 2^{n-1} \binom{n+2}{3} x^n = \frac{-2^{n-2}(n+1)(n+2)(n+6)x^n}{3}\]
    and the lemma is proven.
\end{proof}

This lemma will be used in the proof of the following theorem.
The first formula found in Equation~(\ref{eq:F3_1}) is the one that will prove by explicitly counting matrices giving elements of $\Gr_{2,n}^{\geq 0}(\F_3)$.
This formula will then be used to derive further expressions for $|\Gr_{2,n}^{\geq 0}(\F_3)|$.

\begin{theorem}
For any $n \geq 2$
\begin{align}
    |\Gr_{2,n}^{\geq 0}(\F_3)| &= \sum_{1 \leq i < j < \ell < m \leq n} 2^{n-i-3} + \sum_{1 \leq i < j < \ell \leq n} 2\cdot 2^{n-i-2} + \sum_{1 \leq i < j \leq n} 2^{n-i-1}\label{eq:F3_1}\\
    |\Gr_{2,n}^{\geq 0}(\F_3)| &= \sum_{i \geq 1} 2^{n-i-3} \binom{n-i}{3} + \sum_{i \geq 1} 2^{n-i-1} \binom{n-i}{2} + \sum_{i \geq 1} 2^{n-i-1} \binom{n-i}{1}\label{eq:F3_2}\\
    |\Gr_{2,n}^{\geq 0}(\F_3)| &= \frac{2^{n-4}(n-1)n(n+4)}{3}\label{eq:F3_3}\\
    |\Gr_{2,n}^{\geq 0}(\F_3)| &= [x^{n-2}]\left(\frac{1-x}{(1-2x)^4}\right)\label{eq:F3_4}\\
    |\Gr_{2,n}^{\geq 0}(\F_3)| &= [x^{n-2}]\left(-T_{n+4}(x)\right)\label{eq:F3_5}\\
    |\Gr_{2,n}^{\geq 0}(\F_3)| &= [x^{n-2}t^{n+4}]\left(\frac{-(1-tx)}{1-2tx+t^2}\right)\label{eq:F3_6}
\end{align}
all give the point count of the totally nonnegative finite Grassmannian with $k=2$ over $\F_3$.
\label{thm:2nF3}
\end{theorem}

\begin{proof}
We will work to prove Equation~(\ref{eq:F3_1}) in the theorem from which the others will follow without much work.
Let $A$ be a $2 \times n$ matrix in RREF so that $[A] \in \Gr^{\geq 0}_{2,n}(\F_3)$.
Note that $\F_3 = \{-1,0,1\}$ and $0$ and $1$ are the only nonnegative elements.
We will let $1 \leq i < j \leq n$ denote the pivots of $A$.
Since
\[\det\left(\begin{bsmallmatrix} 1 & * \\ 0 & -1 \end{bsmallmatrix}\right) = \det\left(\begin{bsmallmatrix} 0 & 1 \\ 1 & * \end{bsmallmatrix}\right) = -1\]
it follows that $\begin{bsmallmatrix} * \\ -1 \end{bsmallmatrix}$ and $\begin{bsmallmatrix} 1 \\ *\end{bsmallmatrix}$ cannot occur as columns of $A$ after $i$ and $j$ respectively.

Now we discuss further restrictions on the form of the matrix.
Note that we have 
\[\det\left(\begin{bsmallmatrix} -1 & * \\ 0 & 1 \end{bsmallmatrix}\right) = \det\left(\begin{bsmallmatrix} -1 & 0 \\ * &  1 \end{bsmallmatrix}\right) = -1\]
that need to be avoided.
This means that neither $\begin{bsmallmatrix} -1 \\ 0\end{bsmallmatrix}$ nor $\begin{bsmallmatrix} -1 \\ 1\end{bsmallmatrix}$ can occur before any $\begin{bsmallmatrix} 0 \\ 1\end{bsmallmatrix}$ while neither $\begin{bsmallmatrix} 0 \\ 1\end{bsmallmatrix}$ nor $\begin{bsmallmatrix} -1 \\ 1\end{bsmallmatrix}$ can occur after $\begin{bsmallmatrix} -1 \\ 0\end{bsmallmatrix}$ in the matrix.

We are now ready to summarize the form the matrix $A$ can take in order to obtain our count.
Let $1 \leq i < j < \ell < m$ where $i$ and $j$ denote the pivots and the last occurrence (distinct from the pivot $j$) and first occurrence of $\begin{bsmallmatrix} 0 \\ 1\end{bsmallmatrix}$ and $\begin{bsmallmatrix} -1 \\ 0\end{bsmallmatrix}$ are denoted by $\ell$ and $m$ respectively.
We are assuming both $\begin{bsmallmatrix} 0 \\ 1\end{bsmallmatrix}$ and $\begin{bsmallmatrix} -1 \\ 0\end{bsmallmatrix}$ occur, and if this is the case then all instances of the former must come before the latter.
When one or both of $\ell$ and $m$ are not present, the count is handled similarly, and we will discuss this later.
Now we must examine the choices for all columns which are not $i$, $j$, $\ell$, or $m$.
The only choice for columns before $i$ is $\begin{bsmallmatrix} 0 \\ 0\end{bsmallmatrix}$ while $\begin{bsmallmatrix} 0 \\ 0\end{bsmallmatrix}$ or $\begin{bsmallmatrix} 1 \\ 0\end{bsmallmatrix}$ describe all $j-i-1$ choices between $i$ and $j$.
For columns between $j$ and $\ell$, $\begin{bsmallmatrix} 0 \\ 0\end{bsmallmatrix}$ or $\begin{bsmallmatrix} 0 \\ 1\end{bsmallmatrix}$ cover all $\ell - j - 1$ possible choices.
Between columns $\ell$ and $m$, $\begin{bsmallmatrix} 0 \\ 0\end{bsmallmatrix}$ or $\begin{bsmallmatrix} -1 \\ 1\end{bsmallmatrix}$
are the choices available.
After column $m$ we have $\begin{bsmallmatrix} 0 \\ 0\end{bsmallmatrix}$ or $\begin{bsmallmatrix} -1 \\ 0\end{bsmallmatrix}$ for each of the $n - m$ choices.
In total this gives us
\[2^{j-i-1} 2^{\ell-j-1} 2^{m - \ell - 1} 2^{n-m} = 2^{n-i-3}\]
such matrices for fixed $1 \leq i < j < \ell < m \leq n$.
Summing over all possible choices for $1 \leq i < j < \ell < m \leq n$ gives the first term in the theorem.

If exactly one of $\ell$ or $m$ is missing meaning that after the pivot $j$ one of $\begin{bsmallmatrix} 0 \\ 1\end{bsmallmatrix}$ or  $\begin{bsmallmatrix} -1 \\ 0\end{bsmallmatrix}$ does not occur, then we have
\[\sum_{1 \leq i < j < \ell \leq n} 2^{n-i-2} = \sum_{1 \leq i < j < m \leq n} 2^{n-i-2}\]
such matrices.
Doubling this quantity, because of the choice of which column to avoid, gives the second term in the theorem.
The last term comes from when neither $\ell$ nor $m$ is present and is computed similarly.
Summing over all these cases gives Equation~(\ref{eq:F3_1}).

For Equation~(\ref{eq:F3_2}) in the theorem, we use the formula in Equation~(\ref{eq:F3_1}) where we see the terms in the summations only depend on $n$ and $i$.
Hence, in the first term once the first pivot $i$ is chosen we may choose any distinct $j,\ell,m \in \{i+1, i+2, \dots, n\}$  with $j < \ell < m$ contributing $2^{n-i-3}$ each time.
There are the evidently $\binom{n-i}{3}$ ways to choose $j$,$k$, and $\ell$. The other terms are similar.
One way Equation~(\ref{eq:F3_3}) can be obtained from Equation~(\ref{eq:F3_2}) is by using automated methods (e.g., Maple~\cite{maple} or Mathematica~\cite{Mathematica}) for summing and simplifying sums with binomial coefficients.
Letting $b_n = [x^n]\left(\frac{1-x}{(1-2x)^4}\right)$ for $n \geq 0$, since $(1-2x)^4 = 1 - 8x + 24x^2 - 32x^3 + 16x^4$ it follows that we have the recurrence $b_n - 8b_{n-1} + 24b_{n-2} - 32b_{n-3} + 16b_{n-4} = 0$ for $n \geq 4$.
Thus, Equation~(\ref{eq:F3_4}) is implied by Equation~(\ref{eq:F3_3}) and Lemma~\ref{lem:rec} after accounting for the initial conditions and shift of indices.
Equation~(\ref{eq:F3_5}) is obtained from Lemma~\ref{lem:tile} and Equation~(\ref{eq:F3_3}) again with indices shifted by $2$.
Equation~(\ref{eq:F3_6}) follows by making use of the known generating function for the Chebyshev polynomials of the first kind.
\end{proof}

\begin{remark}
The enumeration found in Theorem~\ref{thm:2nF3} is a simple shift of seqeunce~A001794 in the OEIS~\cite{oeis}.
Starting from the sequence given by~(\ref{eq:F3_5}), the generating function in~(\ref{eq:F3_4}) was found by Plouffe~\cite{plouffe}.
\end{remark}

\subsection{The field $\F_5$}
In this subsection, we consider the field $\F_5$ where we have $\F_5^{>0} = \{1,4\}$ and $\F_5^{<0} = \{2,3\}$.
Hence, in $\F_5$ we have $4=-1 > 0$.
We first have a lemma on $2 \times 2$ determinants of matrices of positive elements in $\F_5$.

\begin{lemma}
Given $a,b \in \F_5^{> 0}$, then for $c,d \in \F_5^{>0}$ we have $ad-bc \geq 0$ if and only if $a=c$ and $b=d$ or else $a=-c$ and $b=-d$.
\end{lemma}
\begin{proof}
Since $\F_5^{>0} = \{1,4\}$ the lemma can be verified by finitely many computations.
\label{lem:2nF5}
\end{proof}

\begin{theorem}
For any $n \geq 2$
\begin{align}
    |\Gr_{2,n}^{\geq 0}(\F_5)| &= \sum_{1 \leq i < j \leq n} 3^{j-i-1} 5^{n-j} + \sum_{1 \leq i < j < \ell \leq n} 4 \cdot 3^{j-i-1} 5^{\ell-j-1} 7^{n-\ell}\label{eq:F5_1}\\\
     |\Gr_{2,n}^{\geq 0}(\F_5)| &= \frac{2\cdot 7^n - 3 \cdot 5^n + 1}{24}\label{eq:F5_2}\\
    |\Gr_{2,n}^{\geq 0}(\F_5)| &= [x^{n-2}] \left(\frac{1}{(1-x)(1-5x)(1-7x)}\right)\label{eq:F5_3}
\end{align}
all give the point count of the totally nonnegative finite Grassmannian with $k=2$ over $\F_5$.
 \label{thm:2nF5}   
\end{theorem}
\begin{proof}
    Let $A$ be a $2 \times n$ matrix in RREF over $\F_5$ so that $[A] \in \Gr^{\geq 0}_{2,n}(\F_5)$.
    Let $1 \leq i < j \leq n$ denote the columns of $A$ which are pivots.
    There are exactly two pivots since $A$ is of full rank.
    We find that all entries of $A$ must be from $\F_5^{\geq 0}$.
    This is because given a column $\begin{bsmallmatrix} a \\ b \end{bsmallmatrix}$ of $A$, the minor when paired with the pivot $\begin{bsmallmatrix} 1 \\ 0 \end{bsmallmatrix}$ will be $b$, while the minor when paired with the pivot $\begin{bsmallmatrix} 0 \\ 1 \end{bsmallmatrix}$ will be $\pm a$.
    Since $[A] \in \Gr^{\geq 0}_{k,n}(\F_5)$ we must have $b \geq 0$ and $\pm a \geq 0$.
    From being in $\F_5$ it implies $a \geq 0$ is equivalent to $-a \geq 0$.

    First, assume $A$ does not contain any column $\begin{bsmallmatrix} a \\ b \end{bsmallmatrix}$ with $a \neq 0$ and $b \neq 0$.
    All columns before $i$ must be $\begin{bsmallmatrix} 0 \\ 0 \end{bsmallmatrix}$.
    Columns between $i$ and $j$ must be of the form $\begin{bsmallmatrix} a \\ 0 \end{bsmallmatrix}$ with $a \in \F_5^{\geq 0}$.
    Hence, we have $3^{j-i-1}$ choices for such columns.
    Columns after $j$ can be of the form $\begin{bsmallmatrix} a \\ 0 \end{bsmallmatrix}$ or $\begin{bsmallmatrix} 0 \\ a \end{bsmallmatrix}$ $a \in \F_q^{\geq 0}$.
    For these columns, we have $5^{n-j}$ choices.
    Thus, we have
    \[\sum_{1 \leq i < j \leq n} 3^{j-i-1} 5^{n-j}\]
    such matrices $A$ which do not contain a column $\begin{bsmallmatrix} a \\ b \end{bsmallmatrix}$ with $a \neq 0$ and $b \neq 0$.

    Next, let $\ell$ be the first column where we have $\begin{bsmallmatrix} a \\ b \end{bsmallmatrix}$ with both $a \neq 0$ and $b \neq 0$.
    Then similar to the previous case we have $3^{j-i-1} 5^{\ell-j-1}$ choices for columns before $\ell$.
    There are $4$ choices for column $\ell$.
    Columns after $\ell$ must be $\begin{bsmallmatrix} c \\ d \end{bsmallmatrix}$ such that $c,d, ad-bc \geq 0$.
    By Lemma~\ref{lem:2nF5} there are $2$ choices with both $c \neq 0$ and $d \neq 0$ and there are $5$ choices we either $c=0$ or $d=0$ recalling that all entries must come from $\F^{\geq 0}_5$.
    This gives 
    \[\sum_{1 \leq i < j < \ell \leq n} 3^{j-i-1} 5^{\ell-j-1}\cdot 4 \cdot 7^{n-\ell}\]
    choices.

    This proves Equation~(\ref{eq:F5_1}), and from this we can obtain Equation~(\ref{eq:F5_2}) again by using automated methods (e.g., Maple~\cite{maple} or Mathematica~\cite{Mathematica}).
    Then Equation~(\ref{eq:F5_3}) can be obtained from Equation~(\ref{eq:F5_2}) using the partial fraction decomposition
    \[\frac{1}{(1-x)(1-5x)(1-7x)} = \frac{49}{12(1-7x)} - \frac{25}{8(1-5x)} + \frac{1}{24(1-x)}\]
    thus completing the proof of the theorem.
\end{proof}

\begin{remark}
One can find that the sequence enumerated in Theorem~\ref{thm:2nF5} as sequence A016230 in the OEIS~\cite{oeis} shifted by $2$.
\end{remark}

\section{Examples and Counterexamples}\label{sec:examples}

In this section, we record some examples demonstrating that some properties known about the real totally nonnegative Grassmannian do not hold in the finite field case.

\subsection{Sign variation}
Let $\F$ be a finite field or the real field, then given $v = (v_1, v_2, \dots, v_n) \in \F^n$ we define its \emph{sign variation}, denoted $\var(v)$, to be the number times $v$, viewed as a sequence of elements of $\F$ while ignoring zeros, switches from positive to negative or vice versa.
As examples we have
\begin{align*}
    u &= (1,0,-1,1,-2) \in \mathbb{R}^5 & \var(u) &= 3\\
    v &= (1,0,2,3,4) \in \mathbb{F}_5^5 & \var(v) &= 2
\end{align*}
where $\var(u)$ can be seen from ``sign flips'' while for $\var(v)$ we recall $1 > 0$, $2 < 0$, $3 < 0$, and $4 > 0$ in $\F_5$.
It is known~\cite{GK} that
\[\Gr^{\geq 0}_{k,n}(\mathbb{R}) = \{V \in \Gr_{k,n}(\mathbb{R}) : \var(v) \leq k-1 \text{ for all } v \in V\}\]
in the real field case.
We refer the reader to Karp~\cite{karp} for uses of sign variation in the study of $\Gr^{\geq 0}_{k,n}(\mathbb{R})$.
We now give examples that show inequalities on sign variation fail to characterize totally nonnegative Grassmannian for finite fields.
This should not be surprising since, though it respects multiplication, our notion of positivity in finite fields does not respect addition.

\begin{example}
Consider $[A] \in \Gr^{\geq 0}_{2,4}(\F_3)$ where
\[ A = \begin{bmatrix}1 & 0 & 2 & 2 \\ 0 & 1 & 1 & 0\end{bmatrix}\]
so $u = (1,2,1,2) \in [A]$ is a linear combination of the rows of $A$.
We see that $\var(u) = 3$.
Also consider $[B] \in \Gr_{2,4}(\F_3)$ where 
\[ B = \begin{bmatrix}0 & 1 & 0 & 1 \\ 0 & 0 & 1 & 1\end{bmatrix}\]
so that $[B] \not\in \Gr^{\geq 0}_{2,4}(\F_3)$ since the maximal minor using the last two columns is negative.
However, it can be checked that $\var(v) \leq 1$ for all $v \in [B]$.   
\end{example}

\subsection{Cyclic shifts and fixed points}
Let $\F$ be any field.
We define $\sigma_{k,n}: \F^n \to \F^n$ by $\sigma_{k,n}((v_1, v_2, \dots, v_n)) = (v_2, v_3, \dots, v_n, (-1)^{k-1}v_1)$.
For $V \in \Gr_{k,n}(\F)$ we let $\sigma_{k,n}(V) = \{\sigma_{k,n}(v) : v \in V\}$.
When $\F$ is a finite field or the real field, it is the case that $\sigma_{k,n}(V) \in \Gr^{\geq 0}_{k,n}(\F)$ whenever $V \in \Gr^{\geq 0}_{k,n}(\F)$.
This is because $V$ and $\sigma_{k,n}(V)$ have as a set the same Pl\"ucker coordinates since the index set of a given Pl\"ucker coordinate is just ``rotated'' by the action of $\sigma_{k,n}$.

Karp~\cite{KarpCyc} showed that in $\Gr_{k,n}(\mathbb{C})$ the action of $\sigma_{k,n}$ has exactly $\binom{n}{k}$ fixed points and there is a unique fixed point in $\Gr_{k,n}^{\geq 0}(\mathbb{R})$.
This result uses that $\sigma_{k,n}$ fixes the span of the vector $(1,z,z^2, \dots, z^{n-1})$ where $z$ is an $n$th root of $(-1)^{k-1}$.
In $\F_q$, some of this can be mimicked in our setting, but generally, these results fail.
We now give examples showing there is not always a unique fixed point in $\Gr^{\geq 0}_{k,n}(\F_q)$.

\begin{example}
    The only fixed points of $\sigma_{2,4}$ in $\Gr_{2,4}(\F_5)$ are the row spans of
    \begin{align*}
        \begin{bmatrix} 1 & 0 & 2 & 0 \\ 0  & 1 & 0 & 2\end{bmatrix} && \text{ and } && \begin{bmatrix} 1 & 0 & 3 & 0 \\ 0  & 1 & 0 & 3\end{bmatrix}
    \end{align*}
    neither of which are in $\Gr^{\geq 0}_{2,4}(\F_5)$.
    By Theorem~\ref{thm:sub} we can view the same two matrices as defining fixed points both in $\Gr^{\geq 0}_{2,4}(\F_{25})$.
\end{example}

\begin{example}
    The spans of each of the vectors
    \begin{align*}
        (1, 1, 1, 1, 1, 1) && (1, 3, 9, 1, 3, 9) && (1, 4, 3, 12, 9, 10)\\[0.65em]
        (1, 9, 3, 1, 9, 3) && (1, 10, 9, 12, 3, 4) && (1, 12, 1, 12, 1, 12)  
    \end{align*}
     are all fixed points of $\sigma_{1,6}$ in $\Gr_{1,6}^{\geq 0}(\F_{13})$.
     Notice each vector here is of the form $(1,z,z^2,z^3,z^4,z^5)$ for a $6$th root of unity $z$ in $\F_{13}$.
     By Lemma~\ref{lem:>0}, $\F_{13}^{> 0}$ is exactly the set of $6$th roots of unity.
\end{example}

\subsection{Matroids}
Given any $V \in \Gr_{k,n}(\F)$ for a field $\F$ one can produce an \emph{$\F$-linear matroid} denoted $M_V$ defined by
\[M_V = \left\{ I \in \binom{[n]}{k} : \Delta_I(V) \neq 0\right\}\]
where the elements $I \in M_V$ are known as \emph{bases} of the matroid.
If $V \in \Gr^{\geq 0}_{k,n}(\R)$, then $M_V$ is what Postnikov~\cite{Pos} calls a \emph{positroid}.
We define an \emph{$\F_q$-postroid} to be an $\F_q$-linear matroid $M_V$ for some $V \in \Gr^{\geq 0}_{k,n}(\F_q)$.
Let us give a small example showing how the theory of $\F_q$-positroids differs from that of $\R$-positroids.

\begin{example}
    Consider the matroid $M = \{\{1,2\},\{2,3\},\{3,4\},\{1,4\}\}$. If $V$ is the row span of
    \[\begin{bmatrix}1 & 0 & a & 0 \\ 0 & 1 & 0 & b \end{bmatrix}\]
    we have $M = M_V$ as an $\F$-linear matroid taking any $a,b \in \F^{\times}$.
    In order for $M = M_V$ to be an $\F$-positroid we need $-a > 0$, $b > 0$, and $ab > 0$.
    This can happen only when $-1 > 0$.
    Thus $M = M_V$ is not an $\F$-postroid when $\F = \R$ nor is it for certain finite fields, e.g.\ $\F = \F_3$.
    However, $M = M_V$ is an $\F$-postroid for some other finite fields, e.g.\ $\F = \F_5$.
\end{example}

One immediate observation we can make about $\F_q$-positroids is the following consequence of Lemma~\ref{lem:altprep}.
Note the \emph{dual matroid} is exactly the matroid whose bases are complements of the bases of the original matroid.

\begin{corollary}
    The dual of an $\F_q$-positroid is an $\F_q$-positroid.
\end{corollary}

Another observation we can make about $\F_q$-positroids is the following consequence of Theorem~\ref{thm:sub}.

\begin{corollary}
    If $s = 2mr$ for an integer $m \geq 1$ and $p$ is an odd prime, then any $\F_{p^r}$-linear matroid is an $\F_{p^s}$-positroid.
\end{corollary}

\section*{Acknowledgements}
The author wishes to thank Nick Ovenhouse for helpful comments, and also thanks the anonymous referee for their careful reading and valuable suggestions.

\bibliographystyle{alphaurl}
\bibliography{refs}

\begin{thebibliography}{CHW24}

\bibitem[BK23]{BlochKarp}
Anthony~M. Bloch and Steven~N. Karp.
\newblock On two notions of total positivity for partial flag varieties.
\newblock {\em Adv. Math.}, 414:Paper No. 108855, 24, 2023.
\newblock \href {https://doi.org/10.1016/j.aim.2022.108855}
  {\path{doi:10.1016/j.aim.2022.108855}}.

\bibitem[BW09]{Cheb}
Arthur~T. Benjamin and Daniel Walton.
\newblock Counting on {Chebyshev} polynomials.
\newblock {\em Math. Mag.}, 82(2):117--126, 2009.
\newblock \href {https://doi.org/10.4169/193009809X468931}
  {\path{doi:10.4169/193009809X468931}}.

\bibitem[CHW24]{posdef}
Joshua Cooper, Erin Hanna, and Hays Whitlatch.
\newblock Positive-definite matrices over finite fields.
\newblock {\em Rocky Mountain J. Math.}, 54(2):423--438, 2024.
\newblock \href {https://doi.org/10.1216/rmj.2024.54.423}
  {\path{doi:10.1216/rmj.2024.54.423}}.

\bibitem[GK50]{GK}
F.~R. Gantmaher and M.~G. Kre\u{\i}n.
\newblock {\em Oscillyacionye matricy i yadra i malye kolebaniya
  mehani\v{c}eskih sistem}.
\newblock Gosudarstv. Izdat. Tehn.-Teor. Lit., Moscow-Leningrad, 1950.
\newblock 2d ed.

\bibitem[GL]{GalashinLam}
Pavel Galashin and Thomas Lam.
\newblock Positroids, knots, and $q,t$-catalan numbers.
\newblock arXiv:2012.09745.
\newblock \href {https://doi.org/10.48550/arXiv.2012.09745}
  {\path{doi:10.48550/arXiv.2012.09745}}.

\bibitem[Hoc75]{Hoch}
Melvin Hochster.
\newblock {\em Topics in the homological theory of modules over commutative
  rings}.
\newblock Conference Board of the Mathematical Sciences Regional Conference
  Series in Mathematics, No. 24. Published for the Conference Board of the
  Mathematical Sciences, Washington, DC; by American Mathematical Society,
  Providence, RI, 1975.
\newblock Expository lectures from the CBMS Regional Conference held at the
  University of Nebraska, Lincoln, Neb., June 24--28, 1974.

\bibitem[Kar17]{karp}
Steven~N. Karp.
\newblock Sign variation, the {G}rassmannian, and total positivity.
\newblock {\em J. Combin. Theory Ser. A}, 145:308--339, 2017.
\newblock \href {https://doi.org/10.1016/j.jcta.2016.08.003}
  {\path{doi:10.1016/j.jcta.2016.08.003}}.

\bibitem[Kar19]{KarpCyc}
Steven~N. Karp.
\newblock Moment curves and cyclic symmetry for positive {G}rassmannians.
\newblock {\em Bull. Lond. Math. Soc.}, 51(5):900--916, 2019.
\newblock \href {https://doi.org/10.1112/blms.12280}
  {\path{doi:10.1112/blms.12280}}.

\bibitem[Lus94]{Lus}
G.~Lusztig.
\newblock Total positivity in reductive groups.
\newblock In {\em Lie theory and geometry}, volume 123 of {\em Progr. Math.},
  pages 531--568. Birkh\"{a}user Boston, Boston, MA, 1994.
\newblock \href {https://doi.org/10.1007/978-1-4612-0261-5\_20}
  {\path{doi:10.1007/978-1-4612-0261-5\_20}}.

\bibitem[{Map}]{maple}
{Maplesoft, a division of Waterloo Maple Inc.}
\newblock Maple.
\newblock URL: \url{https://www.maplesoft.com/}.

\bibitem[{OEI}]{oeis}
{OEIS Foundation Inc.}
\newblock The {O}n-{L}ine {E}ncyclopedia of {I}nteger {S}equences.
\newblock Published electronically at \url{http://oeis.org}.

\bibitem[Plo92]{plouffe}
Simon Plouffe.
\newblock Approximations de s\'eries g\'en\'eratrices et quelques conjectures.
\newblock Master's thesis, Universit\'e du qu\'ebec \'a Montr\'eal, 1992.
\newblock arXiv:0911.4975.

\bibitem[Pos]{Pos}
Alexander Postnikov.
\newblock Total positivity, {G}rassmannians, and networks.
\newblock arXiv:math/0609764.
\newblock \href {https://doi.org/10.48550/arXiv.math/0609764}
  {\path{doi:10.48550/arXiv.math/0609764}}.

\bibitem[Sta89]{unilog}
Richard~P. Stanley.
\newblock Log-concave and unimodal sequences in algebra, combinatorics, and
  geometry.
\newblock In {\em Graph theory and its applications: {E}ast and {W}est
  ({J}inan, 1986)}, volume 576 of {\em Ann. New York Acad. Sci.}, pages
  500--535. New York Acad. Sci., New York, 1989.
\newblock \href {https://doi.org/10.1111/j.1749-6632.1989.tb16434.x}
  {\path{doi:10.1111/j.1749-6632.1989.tb16434.x}}.

\bibitem[{Wol}]{Mathematica}
{Wolfram Research{,} Inc.}
\newblock Mathematica.
\newblock URL: \url{https://www.wolfram.com/mathematica}.

\end{thebibliography}

\end{document}